\documentclass[a4paper,10pt,oneside]{article}

\usepackage{xcolor}
\definecolor{kcolor}{rgb}{0,0.6,0}

\usepackage{tgpagella}
\usepackage[utf8]{inputenc}
\usepackage{latexsym}
\usepackage{amssymb}
\usepackage{amsmath}

\usepackage{hyperref}

\usepackage{pgf}
\usepackage{xxcolor} 

\usepackage{natbib}

\usepackage{mathtools,amsthm}
\usepackage{color}
\usepackage{tikz}
\usepackage{graphicx}
\usepackage{ifpdf}
\usetikzlibrary{calc}
\usepackage{fp} 
\usepackage{enumitem}

\bibpunct[, ]{(}{)}{,}{a}{}{;}

\def\qed{\relax
   \ifmmode
    ~\hfill\Box
   \else
    \unskip\nobreak ~\hfill$\square$%
   \fi \par}

\newcommand{\sep}{$\cdot$ }

\theoremstyle{definition} \newtheorem{cor}{Corollary}
\theoremstyle{definition} \newtheorem{conv}{Convention}
\theoremstyle{definition} 
\theoremstyle{definition} 
\theoremstyle{definition} \newtheorem{rem}{Remark}
\theoremstyle{definition} \newtheorem{lemma}{Lemma}

\newcommand\defequiv{\stackrel{\mathclap{\Delta}}{\equiv}}
\newcommand\eqarrowright{\mathrel{\rotatebox[origin=l]{35}{$\scriptstyle\rightarrow$}}}
\newcommand\eqarrowleft{\mathrel{\rotatebox[origin=r]{-35}{$\scriptstyle\leftarrow$}}}
\newcommand\bhrel{\mathrel{\rotatebox[origin=l]{35}{$\scriptstyle\rightarrow$} \hspace{-2pt}\rotatebox[origin=r]{-35}{$\scriptstyle\leftarrow$}}}
\newcommand\intertrans{\mathrel{\rightleftarrows}}
\newcommand\rename{\stackrel{\mathclap{\emptyset}}{\simeq}}
\newcommand\dbhrel{\stackrel{\mathclap{\emptyset}}{\bhrel}}
\newcommand\defis{\stackrel{\mathclap{\normalfont\mbox{\tiny def}}}{=}}

\makeatletter
\newcommand{\BIG}{\bBigg@{2}}
\newcommand{\BIGG}{\bBigg@{3}}
\newcommand{\vast}{\bBigg@{4}}
\newcommand{\Vast}{\bBigg@{5}}
\makeatother

\definecolor{axcolor}{rgb}{.3,0,.3}

\newcommand{\de}{\stackrel{\text{\tiny def}}{=}}

\usepackage{xspace}

\theoremstyle{definition} \newtheorem{thm}{Theorem}
\theoremstyle{definition} 
\theoremstyle{definition} \newtheorem{defn}{Definition}
 
\begin{document}
	\title{{\small -- preprint -- \\ } On  Generalization of Definitional Equivalence to Languages with Non-Disjoint Signatures}
	\author{Koen Lefever \and Gergely Sz{\'e}kely}
	\date{\today}
	\maketitle

\begin{abstract}
For simplicity, most of the literature introduces the concept of definitional equivalence only to languages with disjoint signatures. In a recent paper, Barrett and Halvorson introduce a straightforward generalization to languages with non-disjoint signatures and they show that their generalization is not equivalent to intertranslatability in general. In this paper, we show that their generalization is not transitive and hence it is not an equivalence relation. Then we introduce the Andr\'eka and N\'emeti generalization as one of the many equivalent formulations for languages with disjoint signatures. We show that the Andr\'eka--N\'emeti generalization is the smallest equivalence relation containing the Barrett--Halvorson generalization and it is equivalent to intertranslatability even for languages with non-disjoint signatures. Finally, we investigate which definitions for definitional equivalences remain equivalent when we generalize them for theories with non-disjoint signatures.

\end{abstract}

\section{Introduction}
\label{intro}
Definitional equivalence\footnote{Definitional equivalence has also been called \emph{logical synonymity} or \emph{synonymy}, e.g., in \citep{deBouvere}, \citep{Friedman} and \citep{visser2015}.} has been studied and used by both mathematicians and philosophers of science as a possible criterion to establish the equivalence between different theories. This concept was first introduced by Montague in \citep{Montague}, but there are already some traces of the idea in \citep{Undecidable}. In philosophy of science, it was introduced by Glymour in \citep{Glymour1970}, \citep{Glymour1977} and \citep{Glymour1980}. Corcoran discusses in \citep{corcoran} the history of definitional equivalence. In \citep[Section 6.3]{BigBook} and \citep[Section 4.3]{Judit}, definitional equivalence is generalized to many-sorted definability, where even new entities can be defined and not just new relations between existing entities. \citep{Glymour}, on which the present paper is partly a commentary, and \citep{Morita} contain more references to examples on the use of definitional equivalence in the context of philosophy of science. 

We have also recently started in \citep{ClassRelKin} to use definitional equivalence to study the exact differences and similarities between theories which are \emph{not} equivalent, in that case classical and relativistic kinematics. In that paper, we showed that there exists a translation of relativistic kinematics into classical kinematics, but not the other way round. We also showed that special relativity extended with a ``primitive ether'' is definitionally equivalent to classical kinematics. Those theories are expressed in the same language, and hence have non-disjoint signatures\footnote{For a variant of this result in which we explicitly made the signatures disjoint, see \citep{diss}.}.

Barrett and Halvorson generalize in \citep[Definition 2]{Glymour} definitional equivalence from \citep[pp. 60-61]{Ho93} for languages having non-disjoint vocabularies in a straightforward way. Then they show that their generalization, which we call here \textit{definitional mergeability} to avoid ambiguity, is not equivalent to intertranslatability in general but only for theories with disjoint signatures. In this paper, we show that definitional mergeability is not an equivalence relation because it is not transitive. Then we recall Andr\'eka and N\'emeti's Definition 4.2 from \citep{definability} which is known to be equivalent to definitional mergeability for languages with disjoint signatures. Then we show that the Andr\'eka--N\'emeti definitional equivalence is the smallest equivalence relation containing definitional mergeablitiy and that it is equivalent to intertranslatability even for theories with languages with non-disjoint signatures. Actually, two theories are definitional equivalent iff there is a theory that is definitionally mergeable to both of them. Moreover, one of these definitional mergers can be a renaming.

Theorem 4.2 of \citep{definability} claims that (i) definitional equivalence, (ii) definitional mergeability, (iii) intertranslatability and (iv) model mergeability (see Definition \ref{def-interdisj} below) are equivalent in case of disjoint signatures. Here, we show that the equivalence of (i) and (iii) and that of (ii) and (iv) hold for arbitrary languages, see Theorems \ref{thm:equiv} and \ref{thm:interdisjoined}. However, since (i) and (ii) are not equivalent by Theorems \ref{thm-not-transitive} and \ref{thm:equivrel},  no other equivalence of extends to arbitrary languages. Finally, we introduce a modification of (iv) that is equivalent to (i) and (iii) for arbitrary languages, see Theorem \ref{thm:intertrans}.
\section{Framework and definitions}
\label{sec:1}
\begin{defn}\label{def-signature}
A \textit{signature}\footnote{In \citep{definability}, a \textit{signature} is called a \textit{vocabulary}. Since this paper is partly a comment on \citep{Glymour}, we will use their terminology, which is also being used in \citep{Ho93} and \citep{Ho97}.} $\Sigma$ is a set of predicate symbols (relation symbols), function symbols, and constant symbols. 
\end{defn}

\begin{defn}\label{def-language}
A \textit{first-order language} $\mathcal{L}$ is a set containing a signature, as well as the terms and formulas which can be constructed from that signature using first-order logic. 
\end{defn}

\begin{rem}\label{rem-simp}
For every theory $T$ which might contain constants and functions, there is another theory $T'$ which is formulated in a language containing only relation symbols and connected to $T$ by all the relations investigated in this paper as candidates for definitional equivalence, see \citep[Proposition~2 and Theorem~1]{Glymour}. Therefore,  here we only consider languages containing only relation symbols.
\end{rem}

\begin{defn}\label{def-sentence}
A \textit{sentence} is a formula without free variables. 
\end{defn}

\begin{defn}\label{def-theory}
A \textit{theory} $T$ is a set of sentences expressed in language $\mathcal{L}$. 
\end{defn}

\begin{conv}
We will use the notations $\Sigma_x$, $\Sigma'$, etc. for the signatures, and $\mathcal{L}_x$, $\mathcal{L}'$, etc. for the languages of respective theories $T_x$, $T'$, etc.
\end{conv}

\begin{defn}\label{def-model}
A \textit{model} $\mathfrak{M} = \langle M,\langle R^{\mathfrak{M}}: R\in \Sigma\rangle \rangle$ of signature $\Sigma$ consists of a non-empty underlying set\footnote{The non-empty underlying set $M$ is also called the \textit{universe}, the \textit{carrier} or the \textit{domain} of $\mathfrak{M}$.} $M$, and for all relation symbols $R$ of $\Sigma$, a relation $R^{\mathfrak{M}} \in M^n$ with the corresponding arity\footnote{The \textit{arity} $n$ is the number of variables in the relation, it is also called the \textit{rank}, \textit{degree}, \textit{adicity} or \textit{valency} of the relation. $M^n$ denotes the Cartesian power of set $M$.}.
\end{defn}

\begin{defn}\label{semantics}
Let $\mathfrak{M}$ be a model, let $M$ be the non-empty underlying set of $\mathfrak{M}$, let $\varphi$ be a formula, let $V$ be the set of variables and let $e:V \to M$ be an evaluation of variables, then we inductively define that $e$ \textit{satisfies $\varphi$ in} $\mathfrak{M}$, in symbols $$\mathfrak{M}\models\varphi[e],$$ as:
\begin{enumerate}
\item \label{semantics1} For predicate $R$, $\mathfrak{M}\models R(x,y,\ldots,z)[e]$ holds if $$\big(e(x), e(y), \ldots , e(z)\big)\in R^\mathfrak{M},$$ 
\item \label{semantics2} $\mathfrak{M}\models (x=y)[e]$ holds if $e(x)=e(y)$ holds,
\item \label{semantics3} $\mathfrak{M}\models \neg \varphi[e]$ holds if $\mathfrak{M}\models \varphi[e]$ does not hold,
\item \label{semantics4} $\mathfrak{M}\models (\psi \land \theta)[e]$ holds if both $\mathfrak{M}\models \psi[e]$ and $\mathfrak{M}\models \theta[e]$ hold,
\item \label{semantics5} $\mathfrak{M}\models \big(\exists y \psi\big)[e]$ holds if there is an element $b \in M$, such that $\mathfrak{M}\models \psi[e']$ if $e'(y)=b$ and $e'(x)=e(x)$ if $x\neq y$.
\end{enumerate}
Let $\bar x$ be the list of all free variables of $\varphi$ and let $\bar a$ be a list of elements of $M$ with the same number of elements as $\bar x$. Then $\mathfrak{M}\models \varphi[\bar{a}]$ iff $\mathfrak{M}$ satsfies\footnote{$\mathfrak{M}\models\varphi[\bar{a}]$ can also be read as $\varphi[\bar{a}]$ \textit{being true in} $\mathfrak{M}$.} $\varphi$ for all (or equivalently some) evaluation $e$ of variables for which $e(\bar x)=\bar a$, i.e., variables in $\bar x$ are mapped to elements of $M$ in $\bar a$ in order. In case $\varphi$ is a sentence, its truth does not depend on evaluation of variables. So that $\varphi$ is true in $\mathfrak{M}$ is denoted by $\mathfrak{M}\models\varphi$. For theory $T$, $\mathfrak{M} \models T$ abbreviates that $\mathfrak{M} \models \varphi$ for all $\varphi\in T$.
\end{defn}

\begin{rem}\label{rem-abbr}
We will use $\varphi \lor \psi$ as an abbreviation for $\neg(\neg \varphi \land \neg \psi)$, $\varphi \rightarrow \psi$ for $\neg \varphi \lor \psi$,  $\varphi \leftrightarrow \psi$ for $(\varphi \rightarrow \psi) \land (\psi \rightarrow \varphi)$ and $\forall x(\varphi)$  for $\neg \big(\exists x(\neg \varphi)\big)$. 
\end{rem}

\begin{defn}\label{def-modelth}
$Mod(T)$ is the class of models of theory $T$,
\[Mod(T)\de\{\mathfrak{M}: \mathfrak{M}\models T\}.\]
\end{defn}

\begin{defn}\label{def-theoryequiv}
Two theories $T_1$ and $T_2$ are logically equivalent, in symbols $$T_1\equiv T_2,$$ iff\footnote{\textit{iff} abbreviates \textit{if and only if}. It is denoted by $\leftrightarrow$ in the object languages (see remark \ref{rem-abbr} above) and by $\iff$ in the meta-language.} they have the same class of models, i.e., $Mod(T_1) = Mod(T_2)$. 
\end{defn}

\begin{defn}\label{def-expldef}
Let $\mathcal{L}\subset \mathcal{L}^+$ be two languages.  An \emph{explicit definition} of an $n$-ary relation symbol $p \in \mathcal{L}^+\setminus\mathcal{L}$ in terms of $\mathcal{L}$ is a sentence of the form
\[\forall x_1 \ldots \forall x_n \big[ p(x_1 , \ldots , x_n ) \leftrightarrow \varphi(x_1 , \ldots , x_n )\big],\]
where $\varphi$ is a formula of $\mathcal{L}$. 
\end{defn}

\begin{defn}\label{def-extension}
A \emph{definitional extension}\footnote{We follow the definition from \citep[Section 4.1, p.36]{definability}, \citep[p.60]{Ho93} and \citep[p.53]{Ho97}. In \citep[Section 3.1]{Glymour}, the logical equivalence relation is not part of the definition.} of a theory $T$ of language $\mathcal{L}$ to language $\mathcal{L}^+$ is a theory $T^+ \equiv T \cup \Delta$, where $\Delta$ is a set of explicit definitions in terms of language $\mathcal{L}$ for each relation symbol $p\in \mathcal{L}^+\setminus\mathcal{L}$.
In this paper, $$T \eqarrowright T^+ \text{ and } T^+ \eqarrowleft T$$ denote that $T^+$ is a definitional extension of $T$.
\end{defn}

We will use $\Delta_{xy}$ to denote the set of explicit definitions when the signature $\Sigma_y$ of theory $T_y$ is defined in terms of the signature $\Sigma_x$ of theory $T_x$.

\begin{defn}\label{def-defeq}
Two theories $T$, $T'$ are \emph{definitionally equivalent}, in symbols $$T \defequiv T',$$ if there is a chain $T_1,\ldots, T_n$ of theories such that $T=T_1$, $T'=T_n$, and for all $1 \leq i < n$ either $T_i \eqarrowright T_{i+1}$ or $T_i \eqarrowleft T_{i+1}$.
\end{defn}

\begin{rem}\label{rem-inconsistent}
If a theory is consistent, then all theories which are definitionally equivalent to  that theory are also consistent since definitions cannot make consistent theories inconsistent. Similarly, if a theory is inconsistent, then all theories which are definitionally equivalent to that theory are also inconsistent.
\end{rem}

\begin{defn}\label{def-bhrel}
Let $T_1$ and $T_2$ be theories of languages $\mathcal{L}_1$ and $\mathcal{L}_2$, respectively. $T_1$ and $T_2$ are \emph{definitionally mergeable}, in symbols $$T_1 \bhrel T_2,$$ if there is a theory $T^+$ which is a common definitional extension of $T_1$ and $T_2$, i.e., $T_1 \eqarrowright T^+ \eqarrowleft T_2$.
\end{defn}

\begin{rem}\label{rem:special}
From Definition~\ref{def-defeq} and Definition~\ref{def-bhrel}, it is immediately clear that being definitionally mergeable is a special case of being definitionally equivalent.
\end{rem}

Lemma \ref{lem-bhdef} below establishes that our Definition \ref{def-bhrel} of definitional mergeability is equivalent to the definition for definitional equivalence in \citep[Definition 2]{Glymour}.

\begin{lemma}\label{lem-bhdef}
Let $T_1$ and $T_2$ be two arbitrary theories. Then $T_1 \bhrel T_2$ iff there are sets of explicit definitions $\Delta_{12}$ and $\Delta_{21}$ such that $T_{1}\cup \Delta_{12}\equiv T_{2}\cup\Delta_{21}$. 
\end{lemma}
\begin{proof}
Let $T_1 \bhrel T_2$, then there exists a $T^+$ such that $T_{1} \eqarrowright T^+ \eqarrowleft T_{2}$. By the definition of definitional extension, there exist sets of explicit definitions $\Delta_{12}$ and $\Delta_{21}$ such that $T_{1}\cup \Delta_{12} \equiv T^+$ and $T_{2}\cup\Delta_{21} \equiv T^+$, and hence by transitivity of logical equivalence $T_{1}\cup \Delta_{12}\equiv T_{2}\cup\Delta_{21}$.

\begin{sloppypar}
To prove the other direction: let $T_1$ and $T_2$ be theories such that ${T_{1} \cup \Delta_{12} \equiv T_{2}\cup\Delta_{21}}$ for some sets $\Delta_{12}$ and $\Delta_{21}$ of explicit definitions. Let $T^+ = T_1 \cup T_2 \cup \Delta_{12} \cup \Delta_{21}$. Hence $T_1 \cup \Delta_{12} \equiv T^+ \equiv T_2\cup \Delta_{21}$ and $T_{1} \eqarrowright T^+ \eqarrowleft T_{2}$, and therefore $T_1\bhrel T_2$.\qedhere
\end{sloppypar}
\end{proof}

\begin{conv}
If theories $T_1$ and $T_2$ are definitionally mergeable and their signatures are disjoint, i.e., $\Sigma_1 \cap \Sigma_2 = \emptyset$, we write $$T_1 \dbhrel T_2.$$
\end{conv}

\begin{defn}\label{def-interdisj} Theories $T_1$ and $T_2$ are \emph{model mergeable}\footnote{We use the definition from \citep[p. 40, item iv]{definability}, which is a variant of the definition in \citep[p. 56, Remark 0.1.6]{HMT71}.}, in symbols $$Mod(T_1) \bhrel Mod(T_2),$$ iff there is a bijection $\beta$ between $Mod(T_1)$ and $Mod(T_2)$ that is defined along two sets $\Delta_{12}$ and $\Delta_{21}$  of explicit definitions such that if $\mathfrak{M}\in Mod(T_1)$, then
\begin{itemize}
\item  the underlying sets of $\mathfrak{M}$ and $\beta(\mathfrak{M})$ are the same,
\item  the relations in $\beta(\mathfrak{M})$ are the ones defined in $\mathfrak{M}$ according to $\Delta_{12}$ and vice versa, the relations in $\mathfrak{M}$ are the ones defined in $\beta(\mathfrak{M})$ according to $\Delta_{21}$.
\end{itemize}
\end{defn}

\begin{defn}\label{def-translation}
Let $T_1$ and $T_2$ be theories. 
A \emph{translation}\footnote{In \citep{definability}, \citep{diss} and \citep{ClassRelKin}, this is called an \emph{interpretation}, but we again follow the terminology from \citep{Glymour} here.} $tr$ of theory $T_1$ to theory $T_2$ is a map from $\mathcal{L}_1$ to $\mathcal{L}_2$ which 
\begin{sloppypar}
\begin{itemize}
\item maps every $n$-ary relation symbol $p\in \mathcal{L}_1$ to a corresponding formula ${\varphi_p\in\mathcal{L}_2}$ of $n$ with free variables, i.e., $tr\big(p(x_1,\ldots,x_n)\big)$ is $\varphi_p(x_1,\ldots,x_n)$.   
\item preserves the equality, logical connectives, and quantifiers, i.e.,  
\begin{itemize}
\item $tr(x_1=x_2)$ is $x_1=x_2$, 
\item $tr(\neg\varphi)$ is $\neg tr(\varphi)$, 
\item $tr(\varphi\land\psi)$ is $tr(\varphi)\land tr(\psi)$, and 
\item $tr(\exists x \varphi)$ is $\exists x \big(tr(\varphi)\big)$.
\end{itemize}
\item maps consequences of $T_1$ into consequences of $T_2$, i.e., $T_1 \models \varphi$ implies $T_2\models tr(\varphi)$ for all sentence $\varphi\in \mathcal{L}_1$. 
\end{itemize}
\end{sloppypar}
\end{defn}

\begin{rem}\label{mutual}
From \citep{AMNmutual}, we know that $T$ being translatable into $T'$ and $T'$ being translatable into $T$ is not a sufficient condition for $T \defequiv T'$. 
\end{rem}

\begin{defn}\label{def-intertrans}
Theories $T_1$ and $T_2$ are \emph{intertranslatable}\footnote{In \citep[p. 167, Definition 4.3.42]{HMT85}, definitional equivalence is defined as intertranslatability.}, in symbols $$T_1 \intertrans T_2,$$  if there are translations $tr_{12}$ of $T_1$ to $T_2$ and $tr_{21}$ of $T_2$ to $T_1$ such that 
\begin{itemize}
\item $T_1 \models \forall x_1 \ldots \forall x_n \big[\varphi(x_1 , \ldots , x_n ) \leftrightarrow tr_{21}\big(tr_{12} \big(\varphi(x_1 ,\ldots , x_n )\big)\big)\big]$
\item $T_2 \models \forall x_1 \ldots \forall x_n \big[\psi(x_1 , \ldots , x_n ) \leftrightarrow tr_{21}\big(tr_{12}\big(\psi(x_1 ,\ldots , x_n )\big)\big)\big]$
\end{itemize}
for every formulas $\varphi(x_1 ,\ldots , x_n )$ and formula $\psi(x_1 ,\ldots , x_n )$ of languages  $\mathcal{L}_1$ and $\mathcal{L}_2$, respectively.
\end{defn}

For a direct proof that intertranslatability is an equivalence relation, see e.g., \citep[Theorem 1, p. 7]{diss}. This fact also follows from Theorems \ref{thm:equivrel} and \ref{thm:equiv} below.

\begin{defn}\label{def-satisfy}
The \textit{relation defined by formula $\varphi$ in $\mathfrak{M}$} is\footnote{$\|\varphi\|^{\mathfrak{M}}$ is basically the same as the \textit{meaning} of formula $\varphi$ in model $\mathfrak{M}$, see \citep[p. 194 Definition 34 and p. 231 Example 8]{HBPL2}.}:
$$\|\varphi \|^{\mathfrak{M}} \defis \big\{\bar a \in M^n : \mathfrak{M}\models \varphi[\bar{a}]\big\}.$$
\end{defn}

\begin{defn}\label{def-transmod}
For all translations $tr_{12}:\mathcal{L}_1\to\mathcal{L}_2$  of theory $T_1$ to theory $T_2$, let $tr^*_{12}$ be defined as the map that maps model $\mathfrak{M}=\langle M,\ldots\rangle$ of $T_2$ to
$$ tr^*_{12}(\mathfrak{M}) \defis \Big\langle M,\left\langle\|tr_{12}(p_i)\|^{\mathfrak{M}}:p_i\in \Sigma_1\right\rangle\Big\rangle,$$
that is all predicates $p_i$ of $\Sigma_1$ interpreted in model $tr^*_{12}(\mathfrak{M})$ as the relation defined by formula $tr_{12}(p_i)$.
\end{defn}

\begin{lemma}\label{lem-transmod}
Let $\mathfrak{M}$ be a model of language $\mathcal{L}_2$, let $\varphi$ be a formula of language $\mathcal{L}_1$, and let $e:V\to M$ be an evaluation of variables. If $tr_{12}:\mathcal{L}_1\to\mathcal{L}_2$  is translation of $T_1$ to $T_2$, then
$$tr^*_{12}(\mathfrak{M})\models \varphi[e] \iff \mathfrak{M}\models tr_{12}(\varphi)[e]$$
\end{lemma}
\begin{proof}
We are going to prove Lemma~\ref{lem-transmod} by induction on the complexity of $\varphi$. So let us first assume that $\varphi$ is a single predicate $p$ of language $\mathcal{L}_1$.

Let $\bar u$ be the $e$-image of the free variables of $p$. Then $tr^*_{12}(\mathfrak{M})\models p[e] $ holds exactly if $tr^*_{12}(\mathfrak{M})\models p[\bar{u}]$. By Definition \ref{def-transmod}, this holds iff
\begin{equation}\label{eq-trmod}\Big\langle M,\left\langle\|tr_{12}(p_i)\|^{\mathfrak{M}}:p_i\in \Sigma_1\right\rangle\Big\rangle \models p[\bar{u}].
\end{equation}
By Definition \ref{def-satisfy}, $\|tr_{12}(p)\|^{\mathfrak{M}} = \big\{\bar a \in M^n : \mathfrak{M}\models tr_{12}(p)[\bar{a}]\big\}$. So \eqref{eq-trmod} is equivalent to $\mathfrak{M}\models tr_{12}(p)[\bar u]$.

 If $\varphi$ is $x = y$, then we should show that $$tr^*_{12}(\mathfrak{M})\models (x = y)[e] \iff \mathfrak{M}\models tr_{12}(x = y)[e].$$
Since translations preserve mathematical equality by Definition \ref{def-translation}, this is equivalent to $$tr^*_{12}(\mathfrak{M})\models (x = y)[e] \iff \mathfrak{M}\models (x = y)[e],$$ which holds because the underlying sets of $tr^*_{12}(\mathfrak{M})$ and $\mathfrak{M}$ are the same and both sides of the equivalence are equivalent to $e(x)=e(y)$ by Definition \ref{semantics}.

Let us now prove the more complex cases by induction on the complexity of formulas. 
\begin{itemize}
\item If $\varphi$ is $\neg \psi$, then we should show that $$tr^*_{12}(\mathfrak{M})\models \neg \psi[e]  \iff \mathfrak{M}\models tr_{12}(\neg \psi)[e].$$  Since $tr_{12}$ is a translation, it preserves (by Definition \ref{def-translation}) the conectives, and therefore this is equivalent to $$tr^*_{12}(\mathfrak{M})\models \neg \psi[e] \iff \mathfrak{M}\models \neg tr_{12}(\psi)[e],$$
which holds by Definition \ref{semantics} Item \ref{semantics3} since we have  
$$tr^*_{12}(\mathfrak{M})\models \psi[e] \iff \mathfrak{M}\models tr_{12}(\psi)[e]$$ by induction.
\item If $\varphi$ is $(\psi \land \theta)$, then we should show that $$tr^*_{12}(\mathfrak{M})\models (\psi \land \theta)[e] \iff \mathfrak{M}\models tr_{12}(\psi \land \theta)[e].$$ Since $tr_{12}$ is a translation, it preserves (by Definition \ref{def-translation}) the conectives, and therefore $tr_{12}(\psi \land \theta)$ is equivalent to $tr_{12}(\psi) \land tr_{12}(\theta)$, and hence the above is equivalent to $$tr^*_{12}(\mathfrak{M})\models (\psi \land \theta)[e] \iff \mathfrak{M}\models \big(tr_{12}(\psi) \land tr_{12}(\theta)\big)[e],$$ which holds by Definition \ref{semantics} Item \ref{semantics4} because both 
$${tr^*_{12}(\mathfrak{M})\models \psi[e] \iff \mathfrak{M}\models tr_{12}(\psi)[e]}$$ and $${tr^*_{12}(\mathfrak{M})\models \theta[e] \iff \mathfrak{M}\models tr_{12}(\theta)[e]}$$ 
hold by induction.
\item If $\varphi$ is $\exists y(\psi)$, then we should show that  $$tr^*_{12}(\mathfrak{M})\models \big(\exists y(\psi)\big)[e] \iff \mathfrak{M}\models tr_{12}\big(\exists y(\psi)\big)[e]$$ holds. Since $tr_{12}$ is a translation, it preserves (by Definition \ref{def-translation}) the quantifiers, and hence this is equivalent to $$tr^*_{12}(\mathfrak{M})\models \big(\exists y(\psi)\big)[e] \iff \mathfrak{M}\models \big(\exists y \big(tr_{12}(\psi)\big)\big)[e].$$
By Definition \ref{semantics} Item \ref{semantics5}, both sides of he equivalence hold exactly if there exists an element $b \in M$ such that
$$tr^*_{12}(\mathfrak{M})\models \psi[e'] \iff \mathfrak{M}\models tr_{12}(\psi)[e'],$$
where $e'(y)=b$ and $e'(x)=e(x)$ if $x\neq y$, which holds by induction because the underlying sets of $tr^*_{12}(\mathfrak{M})$ and $\mathfrak{M}$ are the same.\qedhere
\end{itemize}
\end{proof}

\begin{cor}\label{cor:transmod}
If $tr_{12}:\mathcal{L}_1\to\mathcal{L}_2$  is a translation of $T_1$ to $T_2$, then 
$$tr^*_{12}:Mod(T_2)\to Mod(T_1),$$ 
that is, $tr^*_{12}$  is a map from $Mod(T_2)$ to $Mod(T_1)$.
\end{cor}

\begin{proof}
\begin{sloppypar}
Let $\mathfrak{M}$ be a model of $T_2$ and let $\varphi\in T_1$. We should prove that ${tr^*_{12}(\mathfrak{M})\models \varphi}$. By Lemma~\ref{lem-transmod}, we have that
$$tr^*_{12}(\mathfrak{M})\models \varphi \iff \mathfrak{M}\models tr_{12}(\varphi).$$ 
Hence $tr_{12}(\varphi)$ is true in every model of $T_2$ as we wanted to prove.\qedhere
\end{sloppypar}
\end{proof}

\begin{rem}\label{rem-otherwayround}
Note that while $tr_{12}$ is a translation of $T_1$ to $T_2$, $tr^*_{12}$ translates models the other way round from $Mod(T_2)$ to $Mod(T_1)$. For an example illustrating this for a translation from relativistic kinematics to classical kinematics, see \citep[Chapter 7]{diss} or \citep[Section 7]{ClassRelKin}. 
\end{rem}

\begin{defn}\label{def-intermod} Theories $T_1$ and $T_2$ are \emph{model intertranslatable}, in symbols $$Mod(T_1)\rightleftarrows Mod(T_2),$$ iff there are translations $tr_{12}:\mathcal{L}_1\to\mathcal{L}_2$ of $T_1$ to $T_2$  and $tr_{21}:\mathcal{L}_2\to\mathcal{L}_1$ of $T_2$ to $T_1$, such that $tr^*_{12}:Mod(T_2)\to Mod(T_1)$ and $tr^*_{21}:Mod(T_1)\to Mod(T_2)$ are bijections which are inverses of each other.
\end{defn}

\begin{defn}\label{def-rename}
Theories $T$ and $T'$ are \textit{disjoint renamings} of each other, in symbols $$T \rename T',$$ if their signatures $\Sigma$ and $\Sigma'$ are disjoint, i.e., $\Sigma \cap \Sigma' = \emptyset$, and there is a renaming bijection $R^{\emptyset}_{\Sigma \Sigma'}$ from $\Sigma$ to $\Sigma'$ such that the arity of the relations is preserved and that the formulas in $T'$ are defined by renaming $R^{\emptyset}_{\Sigma \Sigma'}$ of formulas from $T$.\footnote{While bijection $R^{\emptyset}_{\Sigma \Sigma'}$ is defined on signatures, it can be naturally extended to the languages using those signatures. We will use the same symbol $R^{\emptyset}_{\Sigma \Sigma'}$ for that.}
\end{defn}

\begin{rem}\label{rem-rename}
Note that disjoint renaming is symmetric but neither reflexive nor transitive. Also, if $T \rename T'$, then $T \neq T'$, $T \dbhrel T'$, $T \bhrel T'$, $T \defequiv T'$ and $T \intertrans T'$. 
\end{rem}

\section{Properties}
\label{sec:2}
\begin{thm}\label{thm-not-transitive}
Definitional mergeability $\bhrel$ is not transitive. Hence it is not an equivalence relation.
\end{thm}

The proof is based on \citep[Example 5]{Glymour}. Note that the proof relies on the signatures of theories $T_1$ and $T_2$ being non-disjoint.

\begin{proof}
Let $p$ and $q$ be unary predicate symbols. Consider the following theories $T_1$, $T_2$ and $T_3$:
\begin{eqnarray*}
T_1 &=& \{\,\exists! x(x=x),\ \forall x[p(x)]\,\}\\
T_2 &=& \{\,\exists! x(x=x),\ \forall x[\neg p(x)]\,\}\\
T_3 &=& \{\,\exists! x(x=x),\ \forall x[q(x)]\,\}
\end{eqnarray*}
$T_1$ and $T_2$ are not definitionally mergeable, since they do not have a common extension as they contradict each other\footnote{$\exists !$ is an abbreviation for ``there exists exactly one'', i.e., $$\exists ! x \big( \varphi(x) \big) \iff \exists x \Big( \varphi(x) \land \neg \exists y \big( \varphi(y) \land x \neq y \big) \Big).$$}. 

Let us define $T_1^+$ where $q$ is defined in terms of $T_1$ as $p$ and let us define $T_3^+$ where $p$ is defined in terms of $T_3$ as $q$, i.e.,
\begin{eqnarray*}
T_1^+ &=& \{\,\exists! x(x=x),\ \forall x[p(x)],\ \forall x[q(x)\leftrightarrow p(x)] \,\}\\
T_3^+ &=& \{\,\exists! x(x=x),\ \forall x[q(x)],\ \forall x[p(x)\leftrightarrow q(x)]\,\}.
\end{eqnarray*}

Then $T_1$ and $T_3$ are definitionally mergeable because $T_1 \eqarrowright T_1^+$, $T_3 \eqarrowright T_3^+$, and $T_1^+ \equiv T_3^+$. 

Let us now define $T_2^+$ where $q$ is defined in terms of $T_2$ as $\neg p$ and let us define $T_3^\times$ where $p$ is defined in terms of $T_3$ as $\neg q$, i.e.,
\begin{eqnarray*}
T_2^+ &=& \{\,\exists! x(x=x),\ \forall x[\neg p(x)],\ \forall x[q(x)\leftrightarrow \neg p(x)] \,\}\\
T_3^\times &=& \{\,\exists! x(x=x),\ \forall x[q(x)],\ \forall x[p(x)\leftrightarrow \neg q(x)]\,\}.
\end{eqnarray*}

Then $T_2$ and $T_3$ are definitionally mergeable because $T_2 \eqarrowright T_2^+$, $T_3 \eqarrowright T_3^\times$, and $T_2^+ \equiv T_3^\times$.

Therefore, being definitionally mergeable is not transitive and hence not an equivalence relation as $T_1\bhrel T_3\bhrel T_2$ but $T_1$ and $T_2$ are not  definitionally mergeable.\qedhere
\end{proof}

\begin{thm}\label{thm-disj-merge}
If theories $T_1$, $T_2$ and $T_3$ are formulated in languages having disjoint signatures and $T_1 \bhrel T_2$ and $T_2\bhrel T_3$, then $T_1$ and $T_3$ are also mergeable, i.e.,
\[T_1\dbhrel T_2\dbhrel T_3 \text{ and }\Sigma_1\cap\Sigma_3=\emptyset \Longrightarrow T_1 \dbhrel T_3.\]
\end{thm}
\begin{proof}
Let $T_1$, $T_2$ and $T_3$ be theories  such that $\Sigma_1 \cap \Sigma_3 = \emptyset$ and $T_1\dbhrel T_2\dbhrel T_3$.

We have from the definitions of definitional equivalence and definitional extension that there exist sets $\Delta_{12}$, $\Delta_{21}$, $\Delta_{23}$ and $\Delta_{32}$ of explicit definitions, such that
\begin{equation}\label{eq01}
T_1\cup \Delta_{12} \equiv T_2\cup \Delta_{21} 
\text{, i.e., } 
Mod(T_1\cup \Delta_{12}) = Mod(T_2\cup \Delta_{21}),
\end{equation}
and
\begin{equation}\label{eq02}
T_2\cup \Delta_{23}\equiv T_3\cup \Delta_{32} 
\text{, i.e., }
Mod(T_2\cup \Delta_{23})=Mod(T_3\cup \Delta_{32}).
\end{equation}

We want to prove that $T_1\cup \Delta_{12}\cup\Delta_{23}\equiv T_3\cup \Delta_{32}\cup \Delta_{21}$, i.e., 
$$Mod(T_1\cup \Delta_{12}\cup\Delta_{23}) =Mod(T_3\cup \Delta_{32}\cup \Delta_{21}).$$

If one of the theories $T_1$, $T_2$ or $T_3$ is inconsistent, then by Remark \ref{rem-inconsistent}, all of them are inconsistent. In that case $T_1\bhrel T_3$ is true because all statements can be proven ex falso in both theories. Let us for the rest of the proof now assume that all of them are consistent.

Let $\mathfrak{M}\in Mod(T_1\cup \Delta_{12}\cup\Delta_{23})$. Such $\mathfrak{M}$ exists because $\Sigma_1\cap\Sigma_2=\emptyset$ and hence $\Delta_{23}$ cannot make consistent theory $T_1\cup \Delta_{12}$ inconsistent.

\begin{sloppypar}
Then $\mathfrak{M}\models T_1\cup \Delta_{12}\cup\Delta_{23}$. Therefore $\mathfrak{M}\models T_2\cup \Delta_{21}$ by \eqref{eq01} and also ${\mathfrak{M}\models T_3\cup \Delta_{32}}$ because of \eqref{eq02} and the fact that $\mathfrak{M}\models \Delta_{23}$. Hence ${\mathfrak{M}\models T_3\cup \Delta_{32}\cup \Delta_{21}}$. Consequently, $$Mod(T_1\cup \Delta_{12}\cup\Delta_{23}) \subseteq Mod(T_3\cup \Delta_{32}\cup \Delta_{21}).$$
\end{sloppypar}

An analogous calculation shows that $$Mod(T_1\cup \Delta_{12}\cup\Delta_{23}) \supseteq Mod(T_3\cup \Delta_{32}\cup \Delta_{21}).$$ So $Mod(T_1\cup \Delta_{12}\cup\Delta_{23}) =Mod(T_3\cup \Delta_{32}\cup \Delta_{21})$ and this is what we wanted to prove. \qedhere
\end{proof}

\begin{thm}\label{thm:equivrel}
Definitional equivalence $\defequiv$ is an equivalence relation.
\end{thm}
\begin{proof}
To show that definitional equivalence is an equivalence relation, we need to show that it is reflexive, symmetric and transitive:
\begin{itemize}
\item $\defequiv$ is reflexive because for every theory $T \eqarrowright T$ since the set of explicit definitions $\Delta$ can be the empty set, and hence $T \defequiv T$.
\item $\defequiv$ is symmetric: if $T \defequiv T'$, then there exists a chain $T \ldots T'$ of theories connected by $\equiv$, $\eqarrowright$ and $\eqarrowleft$. The reverse chain $T' \ldots T$ has the same kinds of connections, and hence $T' \defequiv T$.
\item $\defequiv$ is transitive: if $T_1 \defequiv T_2$ and $T_2 \defequiv T_3$, then there exists chains $T_1 \ldots T_2$ and $T_2 \ldots T_3$ of theories connected by $\equiv$, $\eqarrowright$ and $\eqarrowleft$. The concatenated chain $T_1 \ldots T_2 \ldots T_3$ has the same kinds of connections, and hence $T_1 \defequiv T_3$. \qedhere
\end{itemize}
\end{proof}

\begin{lemma}\label{lemma-merge}
If $T_1 \defequiv T_2$, then there exists a chain of definitional mergers such that \[T_1 \bhrel T_a \bhrel \ldots \bhrel T_z \bhrel T_2.\]
\end{lemma}

\begin{proof} 
The finite chain of steps given by Definition~\ref{def-defeq} for definitional equivalence can be extended by adding extra extension steps $\eqarrowright$ or $\eqarrowleft$ wherever needed in the chain because definitional extension is reflexive since the set of explicit definitions $\Delta$ can be the empty set. \qedhere
\end{proof}

\begin{lemma} \label{lemma-rename}
Let $T_a$ and $T_b$ two theories for which  $T_a \bhrel T_b$. Then
\begin{itemize}
\item if $T_{b} \rename  T_{b}'$ and $\Sigma_{a} \cap \Sigma'_{b} = \emptyset$, then $T_a \dbhrel T'_b$ ,
\item if  $T_{a} \rename  T_{a}'$ , $T_{b} \rename  T_{b}'$ and $\Sigma'_{a} \cap \Sigma'_{b} = \emptyset$, then $T'_a \dbhrel T'_b$.
\end{itemize}
\end{lemma}
\begin{proof}
Since $T_a \bhrel T_b$, there are by Lemma \ref{lem-bhdef} sets $\Delta_{ab}$ and $\Delta_{ba}$ of explicite definitions such that $T_a\cup \Delta_{ab}\equiv T_b\cup \Delta_{ba}$:
$$\Delta_{ab}=\left\{\forall \bar x \left[ p(\bar x) \leftrightarrow \varphi_p(\bar x) \right]: p\in \Sigma_b \text{ and } \varphi_p\in \mathcal{L}_a\right\},$$
i.e., $\varphi_p$ is the definition of predicate $p$ from $\Sigma_b$ in language $\mathcal{L}_a$.
$$\Delta_{ba}=\left\{\forall \bar x \left[ q(\bar x) \leftrightarrow \varphi_q(\bar x) \right]: q\in \Sigma_a \text{ and } \varphi_q\in \mathcal{L}_b\right\},$$
i.e., $\varphi_q$ is the definition of predicate $q$ from $\Sigma_a$ in language $\mathcal{L}_b$.
We can now define $\Delta_{ab'}$ and $\Delta_{b'a}$ in the following way:
$$\Delta_{ab'} \defis \left\{\forall \bar x \left[ R^{\emptyset}_{\Sigma_b\Sigma'_b}\big(p\big)(\bar x) \leftrightarrow \varphi_p(\bar x) \right]: p\in \Sigma_b \text{ and } \varphi_p\in \mathcal{L}_a\right\},$$
i.e., in $\Delta_{ab'}$ the renaming $R^{\emptyset}_{\Sigma_b\Sigma'_b}\big(p\big)$ of predicate $p$ from $\Sigma_b$ is defined with the same formula $\varphi_p$ as $p$ was defined in $\Delta_{ab}$.
$$\Delta_{b'a} \defis \left\{\forall \bar x \left[ q(\bar x) \leftrightarrow R^{\emptyset}_{\Sigma_b\Sigma'_b}\big(\varphi_q\big)(\bar x) \right]: q\in \Sigma_a \text{ and } \varphi_q\in \mathcal{L}_b\right\},$$
i.e., in $\Delta_{b'a}$ predicate $q$ from $\Sigma_a$ is defined with the renaming $R^{\emptyset}_{\Sigma_b\Sigma'_b}\big(\varphi_q)$ of the formula $\varphi_q$ that was used  in $\Delta_{ba}$ to define $q$.

Then $T_a\cup \Delta_{ab'}\equiv T'_b\cup \Delta_{b'a}$\,, and hence we have proven that $T_a \dbhrel T'_b$.

Similarly, we can define $\Delta_{a'b'}$ and $\Delta_{b'a'}$ as:
$$\Delta_{a'b'} \defis \left\{\forall \bar x \left[ R^{\emptyset}_{\Sigma_b\Sigma'_b}\big(p\big)(\bar x) \leftrightarrow R^{\emptyset}_{\Sigma_a\Sigma'_a}\big(\varphi_p\big)(\bar x) \right]: p\in \Sigma_b \text{ and } \varphi_p\in \mathcal{L}_a\right\},$$
i.e., in $\Delta_{a'b'}$ the renaming $R^{\emptyset}_{\Sigma_b\Sigma'_b}\big(p\big)$ of predicate $p$ from $\Sigma_b$ is defined with the renaming $R^{\emptyset}_{\Sigma_b\Sigma'_b}\big(\varphi_p)$ of the formula $\varphi_p$ that was used  in $\Delta_{ab}$ to define $p$.
$$\Delta_{b'a'} \defis \left\{\forall \bar x \left[ R^{\emptyset}_{\Sigma_a\Sigma'_a}\big(q\big)(\bar x) \leftrightarrow R^{\emptyset}_{\Sigma_b\Sigma'_b}\big(\varphi_q\big)(\bar x) \right]: q\in \Sigma_a \text{ and } \varphi_q\in \mathcal{L}_b\right\},$$
i.e., in $\Delta_{b'a'}$ the renaming $R^{\emptyset}_{\Sigma_b\Sigma'_b}\big(q\big)$ of predicate $q$ from $\Sigma_a$ is defined with the renaming $R^{\emptyset}_{\Sigma_b\Sigma'_b}\big(\varphi_q)$ of the formula $\varphi_q$ that was used  in $\Delta_{ba}$ to define $q$.

Then $T'_a\cup \Delta_{a'b'}\equiv T'_b\cup \Delta_{b'a'}$, and hence we have proven that $T'_a \dbhrel T'_b$. \qedhere
\end{proof}

\begin{thm}\label{thm:renaming} Theories $T_1$ and $T_2$ are definitionally equivalent iff there is a theory $T'_{2}$ which is the disjoint renaming of $T_2$ to a signature which is also disjoint from the signature of $T_1$ such that $T'_{2}$ and $T_1$ are definitionally mergeable, i.e.,
\[T_1\defequiv T_2 \iff \exists T' \big[ T_1 \dbhrel T'_{2} \text{ and } T'_{2}\rename T_2 \big].\] 
\end{thm}

\begin{proof}
Let $T_1$ and $T_2$ be definitional equivalent theories. From Lemma \ref{lemma-merge}, we know that there exists a finite chain of definitonal mergers 
\[T_1 \bhrel T_a \bhrel \ldots \bhrel T_z \bhrel T_2.\] 

For all $x$ in $\{a,\ldots, z, 2\}$, let $T'_x$ be a renaming of $T_x$ such that $\Sigma_1 \cap \Sigma'_x=\emptyset$ and for all $y$ in $\{a,\ldots z, 2\}$, if $x \neq y$ then $\Sigma'_x\cap \Sigma'_y=\emptyset$.

By Lemma \ref{lemma-rename}, $T'_a,\ldots,T'_z,T'_2$ is another chain of mergers from $T_1$ to $T_2$
\[T_1\dbhrel T'_a\dbhrel\ldots T'_z\dbhrel T'_2\rename T_2,\]
where all theories in the chain have signatures which are disjoint from the signatures of all the other theories in the chain, except for $T_1$ and $T_2$ which may have signatures which are non-disjoint.

By Theorem \ref{thm-disj-merge}, the consecutive mergers from $T_1$ to $T'_2$ can be compressed into one merger. So $T_1 \dbhrel T'_2 \rename T_2$ and this is what we wanted to prove.

\begin{sloppypar}
To show the converse direction, let us assume that $T_1$ and $T_2$ are such theories that there is a disjoint renaming theory $T'_{2}$ of $T_2$ for which $T_1\bhrel T'_{2}$. As $T'_{2}$ is a disjoint renaming of $T_{2}$, we have by Remark \ref{rem-rename} that $T'_{2}\dbhrel T_2$. Therefore, there is a chain $T^+$, $T^{\times}$ of theories such that ${T_1\eqarrowright T^+\eqarrowleft T'_{2} \eqarrowright T^{\times}\eqarrowleft T_2}$. Hence $T_1\defequiv T_2$.\qedhere
\end{sloppypar}
\end{proof}

\begin{cor}\label{cor:twostep}
Two theories are definitionally equivalent iff they can be connected by two definitional mergers:
$$T_1\defequiv T_2 \iff \exists T(T_1 \dbhrel T \dbhrel T_2).$$
Consequently, the chain $T_1,\ldots, T_n$ in Definition~\ref{def-defeq} can allways be choosed to be at most length four.
  \end{cor}
\begin{proof}
This follows immediately from Theorem \ref{thm:renaming} and Remark \ref{rem-rename}.\qedhere
\end{proof}

\begin{thm}\label{thm-closure}
Definitional equivalence is the finest equivalence relation containing definitional mergeability. In fact  $\defequiv$ is the transitive closure of relation $\bhrel$.
\end{thm}

\begin{proof}
From Remark~\ref{rem:special}, we know that $\defequiv$ is an extension of $\bhrel$. To prove that $\defequiv$ is the transitive closure of $\bhrel$, it is enough to show that $T_1\defequiv T_2$ holds if there is a chain $T_1',\ldots, T_n'$ of theories such that $T_1=T_1'$, $T_2=T_n'$, and $T_i' \bhrel T_{i+1}'$ for all $1 \leq i < n$. By Theorem~\ref{thm:renaming}, there is a theory $T'$  such that $T_1 \bhrel T'\rename T_2$. By Remark~\ref{rem-rename}, $T_1 \bhrel T'\bhrel T_2$ which proves our statement.\qedhere
\end{proof}

It is known that, for languages with disjoint signatures, being definitionally mergeable and intertranslatability are equivalent, see e.g., \citep[Theorems 1 and 2]{Glymour}.
Now we show that, for languages with disjoint signatures, definitional equivalence also coincides with these concepts, i.e.:
\begin{thm} \label{disjoint-case}
Let $T$ and $T'$ be two theories formulated in languages with disjoint signatures. Then
\[T \defequiv T' \iff T \dbhrel T' \iff T \intertrans T'.\]
\end{thm}
\begin{proof}
Since $T \dbhrel T' \iff T \intertrans T'$ is proven by \citep[Theorems 1 and 2]{Glymour}, we only have to prove that $T \defequiv T' \iff T \dbhrel T'$.

Let theories $T$ and $T'$ be definitionally equivalent theories with disjoint signatures $\Sigma \cap \Sigma' = \emptyset$. Since they are definitionally equivalent, there exists, by Theorem \ref{thm:renaming} a chain which consists of a single mergeability and a renaming step between $T$ and $T'$. Since $T$ and $T'$ are disjoint, and since renaming by Remark \ref{rem-rename} is also a disjoint merger, these two steps can by Theorem \ref{thm-disj-merge} be reduced to one step $T \dbhrel T'$, and this is what we wanted to prove.

The converse direction follows straightforwardly from the definitions. \qedhere
\end{proof}

\begin{thm}\label{thm:interdisjoined}
Let $T_1$ and $T_2$ be arbitrary theories, then $T_1$ and $T_2$ are mergeable iff they are model mergeable, i.e.,
$$T_1 \bhrel T_2 \Longleftrightarrow Mod(T_1) \bhrel Mod(T_2)$$
\end{thm}

\begin{proof} Let $T_1$ and $T_2$ be arbitrary theories.

\begin{sloppypar}
Let us first assume that $T_1 \bhrel T_2$ and prove that $Mod(T_1) \bhrel Mod(T_2)$. We know from Lemma \ref{lem-bhdef} that there exist sets of explicit definitions $\Delta_{12}$ and $\Delta_{21}$ such that $T_{1}\cup \Delta_{12}\equiv T_{2}\cup\Delta_{21}$. Therefore, by Definition \ref{def-theoryequiv}, ${Mod(T_{1}\cup \Delta_{12}) = Mod(T_{2}\cup\Delta_{21})}$.
We construct map $\beta$ between $Mod(T_1)$ and $Mod(T_2)$ by extending models of $T_1$ with the explicit definitions in $\Delta_{12}$, which since $Mod(T_{1}\cup \Delta_{12}) = Mod(T_{2}\cup\Delta_{21})$ will be a model of $T_{1}\cup \Delta_{12}$, and then by taking the reduct to the language of $T_2$. The inverse map $\beta^{-1}$ can be constructed in a completely analogous manner. $\beta$ is a bijection since it has an inverse defined for every model of $T_2$. Through this construction, the relations in $\beta(\mathfrak{M})$ are the ones defined in $\mathfrak{M}$ according to $\Delta_{12}$ and vice versa, the relations in $\mathfrak{M}$ are the ones defined in $\beta(\mathfrak{M})$ according to $\Delta_{21}$, and clearly the underlying set of $\mathfrak{M}$ and $\beta(\mathfrak{M})$ are the same. Hence $Mod(T_1) \bhrel Mod(T_2)$.
\end{sloppypar}

Let us now assume that $Mod(T_1) \bhrel Mod(T_2)$ and prove that $T_1 \bhrel T_2$. We know by Definition \ref{def-interdisj} that there is a bijection $\beta$ between $Mod(T_1)$ and $Mod(T_2)$ that is defined along two sets $\Delta_{12}$ and $\Delta_{21}$  of explicit definitions such that if $\mathfrak{M}\in Mod(T_1)$, then
\begin{itemize}
\item  the underlying set of $\mathfrak{M}$ and $\beta(\mathfrak{M})$ are the same,
\item  the relations in $\beta(\mathfrak{M})$ are the ones defined in $\mathfrak{M}$ according to $\Delta_{12}$ and vice versa, the relations in $\mathfrak{M}$ are the ones defined in $\beta(\mathfrak{M})$ according to $\Delta_{21}$.
\end{itemize}
Any model of both $T_{1}\cup\Delta_{12}$ and $T_{2}\cup\Delta_{21}$ can be obtained by listing the relations of $\mathfrak{M}$ and $\beta(\mathfrak{M})$ together over the common underlying set $M$. Therefore, $Mod(T_{1}\cup \Delta_{12}) = Mod(T_{2}\cup\Delta_{21})$, and thus by Definition \ref{def-theoryequiv}, $T_{1}\cup \Delta_{12}\equiv T_{2}\cup\Delta_{21}$. Consequently, $T_1 \bhrel T_2$.\qedhere
\end{proof}

\begin{thm}\label{thm:equiv}
Let $T_1$ and $T_2$ be arbitrary theories. Then $T_1$ and $T_2$ are definitionally equivalent iff they are intertranslatable, i.e.,  
\[T_1 \defequiv T_2 \iff T_1 \intertrans T_2.\]
\end{thm}
\begin{proof}
Let us first assume that $T_1 \defequiv T_2$. Let $T'$ be a disjoint renaming of $T_2$ to a signature which is also disjoint from the signature of $T_1$. By Remark~\ref{rem-rename} and the transitivity of $\defequiv$, we have $T_1\defequiv T'\defequiv T_2$. By Theorem~\ref{disjoint-case}, $T_1\intertrans T'\intertrans T_2$. Consequently, $T_1 \intertrans T_2$ because relation $\intertrans$ is transitive.

To prove the converse, let us assume that $T_1 \intertrans T_2$. Let $T'$ again be a disjoint renaming of $T_2$ to a signature which is also disjoint from the signature of $T_1$. By Remark~\ref{rem-rename} and the transitivity of $\intertrans$, we have $T_1\intertrans T'\intertrans T_2$. By Theorem~\ref{disjoint-case}, $T_1\defequiv T'\defequiv T_2$. Consequently, $T_1 \defequiv T_2$ because relation $\defequiv$ is transitive.\qedhere
\end{proof}

\begin{thm}\label{thm:intertrans}
Let $T_1$ and $T_2$ be arbitrary theories, then $T_1$ and $T_2$ are intertranslatable iff their models are intertranslatable, i.e.,
$$T_1 \intertrans T_2 \Longleftrightarrow Mod(T_1) \intertrans Mod(T_2)$$
\end{thm}

\begin{proof} Let $T_1$ and $T_2$ be arbitrary theories. If $T_1$ or $T_2$ is inconsistent, then they are by Remark \ref{rem-inconsistent} both inconsistent, $Mod(T_1)$ and $Mod(T_2)$ are empty classes, and the theorem is trivially true. Let's now for the rest of the proof assume that both $T_1$ and $T_2$ are consistent theories and hence that both $Mod(T_1)$ and $Mod(T_2)$ are not empty.

Let us first assume that $T_1 \intertrans T_2$ and prove that $Mod(T_1) \intertrans Mod(T_2)$, i.e., that there exist $tr^*_{12}:Mod(T_2)\to Mod(T_1)$ and $tr^*_{21}:Mod(T_1)\to Mod(T_2)$ which are bijections and which are inverses of each other.

Let $\mathfrak{M}$ be a model of $T_1$, then
$$\mathfrak{M} \models \forall x_1 \ldots \forall x_n \big[\varphi(x_1 , \ldots , x_n ) \leftrightarrow tr_{21}\big(tr_{12} \big(\varphi(x_1 ,\ldots , x_n )\big)\big)\big].$$

By Definition \ref{semantics} and Remark \ref{rem-abbr}, this is equivalent to
$$\mathfrak{M} \models \varphi[e] \iff \mathfrak{M} \models tr_{21}(tr_{12}(\varphi))[e]$$ for all evaluations $e:V\to M$.

By applying Lemma \ref{lem-transmod} twice,
$$\mathfrak{M}\models tr_{21}(tr_{12}(\varphi))[e] \iff tr^*_{21}(\mathfrak{M})\models tr_{12}(\varphi)[e] \iff tr^*_{12}(tr^*_{21}(\mathfrak{M}))\models \varphi[e].$$

Consequently,
$$\mathfrak{M}\models \varphi[e] \iff tr^*_{12}(tr^*_{21}(\mathfrak{M}))\models \varphi[e].$$

Since $M$ is the underlying set of both $\mathfrak{M}$ and $tr^*_{12}(tr^*_{21}(\mathfrak{M}))$, this implies that $\mathfrak{M}=tr^*_{12}(tr^*_{21}(\mathfrak{M}))$.

A completely analogous proof shows that $\mathfrak{N}=tr^*_{21}(tr^*_{12}(\mathfrak{N}))$ for all models $\mathfrak{N}$ of $T_2$.

Consequently, $tr^*_{12}$ and $tr^*_{21}$ are everywhere defined and they are inverses of each other because when we combine them we get the identity, and hence they are bijections, which is what we wanted to prove.

Let us now assume that $Mod(T_1) \intertrans Mod(T_2)$ and prove that $T_1 \intertrans T_2$. By Definition \ref{def-intermod}, we know  that there are bijections $tr^*_{12}$ and $tr^*_{21}$ which are inverses of each other, and thus $\mathfrak{M}=tr^*_{12}(tr^*_{21}(\mathfrak{M}))$ for all models $\mathfrak{M}$ of $T_1$. Since $M$ is the underlying set of $\mathfrak{M}$, and $tr^*_{12}(tr^*_{21}(\mathfrak{M}))$, we have that 
$$\mathfrak{M}\models \varphi[e] \iff tr^*_{12}(tr^*_{21}(\mathfrak{M}))\models \varphi[e].$$ 
From this, by applying Lemma \ref{lem-transmod} twice, we get
$$\mathfrak{M} \models \varphi[e] \iff \mathfrak{M}\models tr_{21}(tr_{12}(\varphi))[e].$$
for all evaluations $e:V\to M$.
By Definition \ref{semantics} and Remark \ref{rem-abbr}, this is equivalent to
$$\mathfrak{M} \models \forall x_1 \ldots \forall x_n \big[\varphi(x_1 , \ldots , x_n ) \leftrightarrow tr_{21}\big(tr_{12} \big(\varphi(x_1 ,\ldots , x_n )\big)\big)\big].$$
A completely analogous proof shows that
$$\mathfrak{N} \models \forall x_1 \ldots \forall x_n \big[\psi(x_1 , \ldots , x_n ) \leftrightarrow tr_{12}\big(tr_{21} \big(\psi(x_1 ,\ldots , x_n )\big)\big)\big],$$
from which follows by Definition \ref{def-intertrans} that $T_1 \intertrans T_2$.\qedhere
\end{proof}

\begin{rem}\label{rem-final}
If we use the notations of this paper,  Theorem 4.2 of \citep{definability} claims, without proof, that (i) definitional equivalence, (ii) definitional mergeability, (iii) intertranslatability and (iv) model mergeability are equivalent in case of disjoint signatures. In this paper, we have not only proven these statements, but we also showed which parts can be generalized to arbitrary languages and which cannot. In detail:
\begin{itemize}
\item item (i) is equivalent to item (iii) by Theorem \ref{disjoint-case}, and we have generalized this equivalence to theories in arbitrary languages by Theorem \ref{thm:equiv},
\item the equivalence of items (ii) and (iv) have been generalized to theories in arbitrary languages by Theorem \ref{thm:interdisjoined},
\item  items (i) and (ii) are indeed equivalent for theories with disjoint signatures by Theorem \ref{disjoint-case}; however, they are not equivalent for theories with non-disjoint signatures by the counterexample in Theorem \ref{thm-not-transitive},
\item  in Definition \ref{def-intermod}, we have introduced a model theoretic counterpart of intertanslatability which, by Theorem \ref{thm:intertrans}, is equivalent to it even for arbitrary languages.
\end{itemize}
\end{rem}

\section{Conclusion}
\label{sec:3}
Since definitional mergeability is not transitive, by Theorem \ref{thm-not-transitive}, and thus not an equivalence relation, the Barrett--Halvorson generalization is not a well-founded criterion for definitional equivalence when the signatures of theories are not disjoint. Contrary to this, the Andr\'eka--N\'emeti generalization of definitional equivalence is an equivalence relation, by Theorem \ref{thm:equivrel}. It is also equivalent to intertranslatability, by Theorem \ref{thm:equiv}, and to model-intertranslatability, by Theorem \ref{thm:intertrans}, even for languages with non-disjoint signatures. Therefore, the Andr\'eka--N\'emeti generalization is more suitable to be used as the extension of definitional equivalence between theories of arbitrary languages. It is worth noting, however, that the two generalizations are really close to each-other since the Andr\'eka--N\'emeti generalization is the transitive closure of the Barrett-Halvorson one, see Theorem \ref{thm-closure}. Moreover, they only differ in at most one disjoint renaming, see Theorems \ref{thm:renaming} and  \ref{disjoint-case}, and as long as we restrict ourselves to theories which all have mutually disjoint signatures, Barrett--Halvorson's definition is transitive by Theorem~\ref{thm-disj-merge}.

\subsection*{Acknowledgements}
The writing of the current paper was induced by questions by Marcoen Cabbolet and Sonja Smets during the public defence of \citep{diss}. We are also grateful to Hajnal Andr\'eka, Mohamed Khaled, Amed\'e Lefever, Istv\'an N\'emeti and Jean Paul Van Bendegem for enjoyable discussions and feedback while writing this paper.



\bibliographystyle{agsm}
\bibliography{LogRel12017} 




\begin{flushright}
KOEN LEFEVER\\
Centre for Logic and Philosophy of Science\\
Vrije Universiteit Brussel\\
\href{mailto:koen.lefever@vub.ac.be}{koen.lefever@vub.be}\\
\url{http://homepages.vub.ac.be/~kolefeve/}

\vspace{.7cm}

GERGELY SZ{\' E}KELY\\
MTA Alfr{\' e}d R{\' e}nyi Institute for Mathematics\\
\href{mailto:szekely.gergely@renyi.mta.hu}{szekely.gergely@renyi.mta.hu}\\
\url{http://www.renyi.hu/~turms/}
\end{flushright}

\end{document}